\documentclass{article}
\usepackage[utf8]{inputenc}

\title{Remarks on Fixed Point Assertions in Digital
Topology, 3}
\author{Laurence Boxer
\thanks{Department of Computer and Information Sciences, Niagara University, NY 14109, USA; and
Department of Computer Science and Engineering, SUNY at Buffalo, Buffalo, NY, USA. email: boxer@niagara.edu}
}

\usepackage{graphicx}
\usepackage{amsmath,amsthm,amssymb,amsfonts}
\usepackage{verbatim}

\usepackage{stmaryrd} 
\usepackage{hyperref}

\usepackage[colorinlistoftodos]{todonotes}

\usepackage{tikz}
\usepackage{tkz-berge, tkz-graph}

\tikzset{vertex/.style={circle,draw,fill,inner sep=0pt,minimum size=1mm}}

\theoremstyle{plain}
\newtheorem{thm}{Theorem}

\newtheorem{prop}[thm]{Proposition}

\newtheorem{cor}[thm]{Corollary}

\newtheorem{remark}[thm]{Remark}

\theoremstyle{definition}

\newtheorem{definition}[thm]{Definition}
\newtheorem{exl}[thm]{Example}

\numberwithin{thm}{section}

\newcommand{\adj}{\leftrightarrow}
\newcommand{\adjeq}{\leftrightarroweq}

\def\N{{\mathbb N}}
\def\Z{{\mathbb Z}}
\def\R{{\mathbb R}}

\begin{document}
\date{}
\maketitle

\begin{abstract}
We continue the work of~\cite{BxSt19} and~\cite{Bx19}, in which are considered papers
in the literature that discuss fixed point assertions
in digital topology. We discuss published assertions 
that are incorrect or incorrectly proven; that are 
severely limited or reduce to triviality under 
``usual" conditions; or that we improve upon.

Key words and phrases: digital image; fixed point;
approximate fixed point
\end{abstract}

\section{Introduction}
The topic of fixed points in digital topology has drawn
much attention in recent papers. The quality of
discussion among these papers is uneven; 
while some assertions have been correct and interesting, others have been incorrect, incorrectly proven, or reducible to triviality.
In~\cite{BxSt19} and~\cite{Bx19}, we have discussed many shortcomings in earlier papers and have offered
corrections and improvements. We continue this work in the current paper.

\section{Preliminaries}
We use $\N$ to represent the natural numbers,
$\Z$ to represent the integers, and $\R$ to represent the reals.

A {\em digital image} is a pair $(X,\kappa)$, where $X \subset \Z^n$ 
for some positive integer $n$, and $\kappa$ is an adjacency relation on $X$. 
Thus, a digital image is a graph.
In order to model the ``real world," we usually take $X$ to be finite,
although there are several papers that consider
infinite digital images. The points of $X$ may be 
thought of as the ``black points" or foreground of a 
binary, monochrome ``digital picture," and the 
points of $\Z^n \setminus X$ as the ``white points"
or background of the digital picture.

\subsection{Adjacencies, connectedness, continuity}
In a digital image $(X,\kappa)$, if
$x,y \in X$, we use the notation
$x \adj_{\kappa}y$ to
mean $x$ and $y$ are $\kappa$-adjacent; we may write
$x \adj y$ when $\kappa$ can be understood. 
We write $x \adjeq_{\kappa}y$, or $x \adjeq y$
when $\kappa$ can be understood, to
mean 
$x \adj_{\kappa}y$ or $x=y$.

The most commonly used adjacencies in the study of digital images 
are the $c_u$ adjacencies. These are defined as follows.
\begin{definition}
Let $X \subset \Z^n$. Let $u \in \Z$, $1 \le u \le n$. Let 
$x=(x_1, \ldots, x_n),~y=(y_1,\ldots,y_n) \in X$. Then $x \adj_{c_u} y$ if 
\begin{itemize}
    \item for at most $u$ distinct indices~$i$,
    $|x_i - y_i| = 1$, and
    \item for all indices $j$ such that $|x_j - y_j| \neq 1$ we have $x_j=y_j$.
\end{itemize}
\end{definition}

\begin{definition}
{\rm \cite{Rosenfeld}}
A digital image $(X,\kappa)$ is
{\em $\kappa$-connected}, or just {\em connected} when
$\kappa$ is understood, if given $x,y \in X$ there
is a set $\{x_i\}_{i=0}^n \subset X$ such that
$x=x_0$, $x_i \adj_{\kappa} x_{i+1}$ for
$0 \le i < n$, and $x_n=y$.
\end{definition}

\begin{definition}
{\rm \cite{Rosenfeld, Bx99}}
Let $(X,\kappa)$ and $(Y,\lambda)$ be digital
images. A function $f: X \to Y$ is 
{\em $(\kappa,\lambda)$-continuous}, or
{\em $\kappa$-continuous} if $(X,\kappa)=(Y,\lambda)$, or
{\em digitally continuous} when $\kappa$ and
$\lambda$ are understood, if for every
$\kappa$-connected subset $X'$ of $X$,
$f(X')$ is a $\lambda$-connected subset of $Y$.
\end{definition}

\begin{thm}
{\rm \cite{Bx99}}
A function $f: X \to Y$ between digital images
$(X,\kappa)$ and $(Y,\lambda)$ is
$(\kappa,\lambda)$-continuous if and only if for
every $x,y \in X$, if $x \adj_{\kappa} y$ then
$f(x) \adjeq_{\lambda} f(y)$.
\end{thm}

\begin{thm}
\label{composition}
{\rm \cite{Bx99}}
Let $f: (X, \kappa) \to (Y, \lambda)$ and
$g: (Y, \lambda) \to (Z, \mu)$ be continuous 
functions between digital images. Then
$g \circ f: (X, \kappa) \to (Z, \mu)$ is continuous.
\end{thm}

\subsection{Fixed, approximate fixed points}
A {\em fixed point} of a function $f: X \to X$ is a
point $x \in X$ such that $f(x) = x$. If
$(X,\kappa)$ is a digital image, an 
{\em almost fixed point} ~\cite{Rosenfeld}
or {\em approximate fixed point}~\cite{BEKLL} of
$f: X \to X$ is a
point $x \in X$ such that $f(x) \adjeq_{\kappa} x$. 

\subsection{Digital metric spaces}
A {\em digital metric space}~\cite{EgeKaraca15} is a triple
$(X,d,\kappa)$, where $(X,\kappa)$ is a digital image and $d$ is a metric on $X$. 
We are not convinced that
this is a notion worth developing; under conditions 
in which a digital image models a ``read world" image, 
$X$ is finite or $d$ is (usually) an $\ell_p$ metric, so that 
$(X,d,\kappa)$ is discrete as a topological space. Typically,
assertions in the literature do not make use of both
$d$ and $\kappa$,
so that this notion has an artificial feel. E.g., for a
discrete topological space, all self-maps are continuous,
although on digital images, self-maps are often not digitally continuous.

We say a sequence $\{x_n\}_{n=0}^{\infty}$ is {\em eventually constant} if for some $m>0$, $n>m$
implies $x_n=x_m$.

\begin{prop}
{\em \cite{Han16}}
\label{eventuallyConst}
Let $(X,d, \kappa)$ be a digital metric space. If for some $a>0$ and all
distinct $x,y \in X$ we have $d(x,y) > a$, then any Cauchy sequence in $X$
is eventually constant, and $(X,d)$ is a complete metric space.
\end{prop}

Note that the hypotheses of Proposition~\ref{eventuallyConst} are satisfied
if $X$ is finite or if $d$ is an $\ell_p$ metric.

\section{Universal functions and AFPP}

A digital image $(X,\kappa)$ has the {\em approximate fixed point property} (AFPP) if every
$\kappa$-continuous $f: X \to X$ has an
approximate fixed point.

We can paraphrase Theorem~3.3 of~\cite{Rosenfeld} as
follows.

\begin{thm}
\label{RosenfeldAFPP}
A digital interval $([a,b]_{\Z}, c_1)$ has the AFPP.
$\Box$
\end{thm}

\begin{definition}
\label{univ}
{\rm \cite{BEKLL}}
Let $(X,\kappa)$ and $(Y,\lambda)$ be digital images.
A $(\kappa, \lambda)$-continuous
function $f: X \to Y$ is {\em universal for $(X, Y)$}
if given a $(\kappa, \lambda)$-continuous function
$g: X \to Y$ such that $g \neq f$,
there exists $x \in X$ such that 
$f(x) \adj_{\lambda} g(x)$.
\end{definition}

It was shown in~\cite{BEKLL} that there is a
relationship between the AFPP and universal functions. In this section, we show there are
advantages in the study of the AFPP to replacing the
notion of universal function with a similar notion of
a ``weakly universal function." This enables us to
make several improvements on results of~\cite{BEKLL}.

The following assertion, one implication of which is 
incorrect, appears as Proposition~5.5
of~\cite{BEKLL}.

\begin{quote}
    Let $(X,\kappa)$ be a digital image. Then
    $(X,\kappa)$ has the AFPP if and only if
    the identity function $1_X$ is universal for
    $(X,X)$.
\end{quote}

The implication of this assertion that is correct
is stated in the following with its proof as given
in~\cite{BEKLL}.

\begin{prop}
\label{univ-AFPP}
Let $(X,\kappa)$ be a digital image. If
    the identity function $1_X$ is universal for
    $(X,X)$, then $(X,\kappa)$ has the AFPP.
\end{prop}

\begin{proof}
If $1_X$ is universal for $(X,X)$, then for
$1_X \neq f: X \to X$, $f$ being 
$\kappa$-continuous, there exists $x \in X$ such that
$f(x)\adj_{\kappa} 1_X(x) = x$. Thus $X$ has 
the AFPP.
\end{proof}

However, the converse of Proposition~\ref{univ-AFPP}
is not generally true, as shown be the following.

\begin{exl}
\label{nonUwAFPP}
Let $f: ([-1,1]_{\Z},c_1) \to ([-1,1]_{\Z},c_1)$
be the map $f(z)=-z$. Then $f$ is $c_1$-continuous, and there is no 
$z \in [-1,1]_{\Z}$ such that $f(z) \adj_{c_1} z$.
Hence $1_{[-1,1]_{\Z}}$ is not a universal 
function for $([-1,1]_{\Z},[-1,1]_{\Z})$.
However, by Theorem~\ref{RosenfeldAFPP}, $([-1,1]_{\Z},c_1)$ has the AFPP. $\Box$
\end{exl}

\begin{definition}
\label{weakUniv}
Let $(X,\kappa)$ and $(Y,\lambda)$ be digital 
images. Let
$f: X \to Y$ be $(\kappa,\lambda)$-continuous. Then
$f$ is a {\em weakly universal function} for
$(X,Y)$ if for every $(\kappa,\lambda)$-continuous
$g: X \to Y$ such that $g \neq f$ there exists
$x \in X$ such that $f(x) \adjeq_{\lambda} g(x)$.
$\Box$
\end{definition}

Notice the difference between Definitions~\ref{univ}
and~\ref{weakUniv}: the former requires
$f(x)$ and $g(x)$ to be adjacent, while the latter
requires $f(x)$ and $g(x)$ to be adjacent or equal.

\begin{prop}
\label{univIsWeakly}
A universal function between digital images is
weakly universal.
\end{prop}

\begin{proof}
This is immediate from Definitions~\ref{univ}
and~\ref{weakUniv}.
\end{proof}

For a graph $G=(V,E)$ ($V$ is the vertex set; $E$
is the edge set), a subset $D$ of $V$ is called
{\em dominating} if for every $v \in V$, either
$v \in D$ or there is a $w \in D$ such that
$\{v,w\} \in E$~\cite{CL}. The following generalizes
a result of~\cite{BEKLL}.

\begin{prop}
Let $(X,\kappa)$ and $(Y,\lambda)$ be digital images.
If $f: X \to Y$ is a weakly universal function, then
$f(X)$ is $\lambda$-dominating in $Y$.
\end{prop}

\begin{proof}
Let $y\in Y$ and let $\tilde{y}: X \to Y$ be
the constant function with image $\{y\}$. Since
$f$ is weakly universal, there exists $x \in X$
such that $f(x) \adjeq_{\lambda} \tilde{y}(x) = y$.
Since $y$ is an arbitrary member of $Y$, 
the assertion follows.
\end{proof}

\begin{thm}
\label{1XweaklyUniv}
The digital image $(X,\kappa)$ has the AFPP if and
only if $1_X$ is weakly universal for $(X,X)$.
$\Box$
\end{thm}

\begin{proof}
$(X,\kappa)$ has the AFPP if and only if given 
a $(\kappa,\kappa)$-continuous $f: X \to X$,
for some $x \in X$ we
have $f(x) \adjeq_{\kappa} 1_X(x) = x$; i.e.,
if and only if $1_X$ is universal.
\end{proof}

The following is suggested by Theorem~5.7
of~\cite{BEKLL}.

\begin{prop}
\label{propPartial}
Let $(W,\kappa)$, $(X,\lambda)$, and $(Y,\mu)$ be
digital images. Let $f: W \to X$ be
$(\kappa,\lambda)$-continuous and let $g: X \to Y$
be $(\lambda, \mu)$-continuous. If $g \circ f$ is
weakly universal, then $g$ is weakly universal.
\end{prop}

\begin{proof}
Let $h: X \to Y$ be $(\lambda, \mu)$-continuous.
Since $g \circ f$ is weakly universal, there exists
$w \in W$ such that 
$g \circ f(w) \adjeq_{\mu} h \circ f(w)$, i.e.,
for $x = f(w)$ we have $g(x) \adjeq_{\mu} h(x)$.
Since $h$ was arbitrarily chosen, the assertion
follows.
\end{proof}

The following is suggested by Theorem~5.8
of~\cite{BEKLL}.

\begin{thm}
\label{compositionWU}
Let $g: (U, \mu) \to (X,\kappa)$ and 
$h: (Y, \lambda) \to (V, \nu)$ be digital
isomorphisms. Let $f: X \to Y$ be
$(\kappa,\lambda)$-continuous. Then the
following are equivalent.

{\rm (1)} $f$ is a weakly universal function for
          $(X,Y)$.
          
{\rm (2)} $f \circ g$ is weakly universal.

{\rm (3)} $h \circ f$ is weakly universal.
\end{thm}

\begin{proof}
(1 implies 2): Let $k: U \to Y$ be
$(\mu, \lambda)$-continuous. Since $f$ is
weakly universal, there exists $x \in X$ such
that $(k \circ g^{-1})(x) \adjeq_{\lambda} f(x)$, 
i.e., for $u=g^{-1}(x)$ we have
\[k(u) = k(g^{-1}(x)) \adjeq_{\lambda} 
(f \circ g)(g^{-1}(x)) = (f \circ g)(u).
\]
Since $k$ is arbitrary, $f \circ g$ is 
weakly universal.

(2 implies 1): This follows from
Proposition~\ref{propPartial}.

(1 implies 3): Let $m: X \to V$ be
$(\kappa,\nu)$-continuous. Since $f$ is
weakly universal, there exists $x \in X$ such
that $(h^{-1} \circ m)(x) \adjeq_{\lambda} f(x)$.
By continuity,
\[ m(x) = h \circ (h^{-1} \circ m)(x) \adjeq_{\nu} h \circ f(x).
\]
Since $m$ is arbitrary, $h \circ f$ is weakly
universal.

(3 implies 1): Let $r: X \to Y$ be
$(\kappa,\lambda)$-continuous. Since $h \circ f$
is weakly universal, there exists $x \in X$ such that
$h \circ f(x) \adjeq_{\nu} h \circ r(x)$. Thus,
\[ f(x) = h^{-1} \circ h \circ f(x)
   \adjeq_{\lambda} h^{-1} \circ h \circ r(x) = r(x).
\]
Since $r$ is arbitrary, $f$ must be weakly universal.
\end{proof}

Corollary~5.9 of~\cite{BEKLL} claims that an
isomorphism $f: (X,\kappa) \to (Y,\lambda)$ is
universal for $(X,Y)$ if and only if $(X,\kappa)$
has the AFPP. Example~\ref{nonUwAFPP} above shows
that this assertion is incorrect. However, we
have the following.

\begin{cor}
Let $f: (X,\kappa) \to (Y,\lambda)$ be an 
isomorphism. The following are equivalent.

{\rm (1)} $f$ is weakly universal for
$(X,Y)$.

{\rm (2)} $(X,\kappa)$ has the AFPP.

{\rm (3)} $(Y,\lambda)$ has the AFPP.
\end{cor}

\begin{proof}
(1) $\Leftrightarrow$ (2): By Theorem~\ref{compositionWU}, $f$ is weakly
universal if and only if $f \circ f^{-1}=1_X$ is 
weakly universal, which, by 
Theorem~\ref{1XweaklyUniv} is true if and only if
$(X,\kappa)$ has the AFPP.

(1) $\Leftrightarrow$ (3): By Theorem~\ref{compositionWU}, $f$ is weakly
universal if and only if $f^{-1} \circ f=1_Y$ is 
weakly universal, which, by 
Theorem~\ref{1XweaklyUniv} is true if and only if
$(Y,\lambda)$ has the AFPP.
\end{proof}

The following generalizes Theorem~5.10
of~\cite{BEKLL} and corrects its proof (stated in
terms of Proposition~5.5 of~\cite{BEKLL}, which, we noted above, is incorrect).
\begin{thm}
Let $(X_i,\kappa_i)$ be
digital images, $1 \le i \le v$. Let
$X = \Pi_{i=1}^v X_i$. If
$(X,NP_v(\kappa_1,\ldots, \kappa_v))$ has the
AFPP, then each $(X_i,\kappa_i)$ has the AFPP.
\end{thm}

\begin{proof}
Let $f_i: X_i \to X_i$ be $(\kappa_i,\kappa_i)$-continuous, $1 \le i \le v$.
Then the product function 
$f = \Pi_{i=1}^v f_i: X \to X$ is
$(NP_v(\kappa_1,\ldots, \kappa_v),NP_v(\kappa_1,\ldots, \kappa_v))$-continuous~\cite{Bx17Gen}. By Theorem~\ref{1XweaklyUniv}, $1_X$ is weakly
universal, so there exists
$p=(x_1,\ldots, x_v) \in X$, $x_i \in X_i$, such that
\[p \adjeq_{NP_v(\kappa_1,\ldots, \kappa_v)} f(p) =
  (f_1(x_1), \ldots, f_v(x_v)),
\]
hence $x_i \adjeq_{\kappa_i} f_v(x_i)$ for all~$i$. Since $f_i$ was taken arbitrarily, the
conclusion follows.
\end{proof}

\section{Digital expansions in~\cite{Jyoti-Rani17}}
The paper~\cite{Jyoti-Rani17} contains several
assertions that are incorrect or incorrectly proven,
limited, or can be improved.

\subsection{Digital expansive mappings}
\begin{definition}
{\rm \cite{Jyoti-Rani17}}
Let $(X,d,\kappa)$ be a complete digital metric 
space. Let $T: X \to X$. If
$d(T(x),T(y)) \ge k \, d(x,y)$ for all $x,y \in X$
and some $k>1$, then $T$ is a
{\em digital expansive mapping}.
\end{definition}

\begin{thm}
{\rm \cite{Jyoti-Rani17}}
\label{expansiveThm}
If $T: X \to X$ is a digital expansive mapping on
complete digital metric space $(X,d,\kappa)$ and
$T$ is onto, then $T$ has a fixed point.
\end{thm}

However, in practice, the hypotheses of 
Theorem~\ref{expansiveThm} often cannot be
satisfied, as shown in the following, which
combines Theorems~4.8 and~4.9 of~\cite{BxSt19}.

\begin{thm}
\label{expansiveLimit}
Let $(X,d,\kappa)$ be a digital metric space of more than one point. If there exist $x_0,y_0 \in X$
such that either
\begin{itemize}
    \item $d(x_0,y_0) = diam \, X > 0$, or
    \item $d(x_0,y_0) = \min \{d(x,y) \, | \,
           x,y \in X, x \neq y\}$,
\end{itemize}
then there is no self-map $T: X \to X$ that is
a digital expansive mapping and is onto.
\end{thm}

In practice, a digital image $(X,\kappa)$ typically
consists of a finite set of more than 1 point; or,
should a metric $d$ be used, it is typically an
$\ell_p$ metric. Under such circumstances,
by Theorem~\ref{expansiveLimit} a digital metric 
space $(X,d,\kappa)$ cannot have a 
self-map that is both a digital 
expansive mapping and onto.

\subsection{1st generalization of expansive mappings}
Theorem~3.4 of~\cite{Jyoti-Rani17} states the
following.

\begin{thm}
\label{expansiveSumLimit}
Let, $(X,d,\kappa)$ be a complete digital metric space and $T: X \to X$ be an onto self map. Let $T$ satisfy
$d(T(x), T(y)) \ge k \, [d(x, T(x)) + d(y,T(y))]$  where $k\ge 1/2$, for all $x,y \in X$.
Then $T$ has a fixed point.
\end{thm}

But Theorem~\ref{expansiveSumLimit} reduces to a
trivial statement, as we see in the following.

\begin{prop}
A map $T$ as in Theorem~\ref{expansiveSumLimit}
must be the identity map.
\end{prop}

\begin{proof}
For such a map, we have
\[ 0 = d(T(x), T(x)) \ge k \, [d(x,T(x)) + d(x, T(x))] = 2k \, d(x, T(x)),
\]
so $d(x, T(x))=0$ for all $x \in X$.
\end{proof}

\subsection{2nd generalization of expansive mappings}
\label{2ndGen}
Theorem~3.5 of~\cite{Jyoti-Rani17} asserts the
following.
\begin{quote}
    Let $(X,d,\kappa)$ be a complete digital metric 
    space and let $T: X \to X$ be onto and 
    continuous. Let
    \[ d(T(x), T(y)) \ge k \mu(x,y) \]
    for all $x,y \in X$, where $k>1$ and
    \[ \mu(x,y) \in \left \{d(x,y), 
    \frac{d(x,T(x)) + d(y,T(y))}{2},
    \frac{d(x,T(y)) + d(y,T(x))}{2}
    \right \}.
    \]
    Then $T$ has a fixed point.
\end{quote}

The argument given as proof for this assertion
has flaws,
including the use in its proof of an assumption
that $k<2$, not stated in the hypotheses; and an 
incorrect application of the triangle inequality (where we need the reverse of the inequality to proceed as the authors have done) in the attempt to reduce Case~3 to Case~2. Thus, the assertion as stated must be regarded as unproven.
Also, the argument given for proof clarifies that the continuity assumption
is of the $\varepsilon-\delta$ type, not digital. In the following, we obtain a version of this assertion
with no continuity assumption, but with an additional assumption about $X$ or $d$ and with greater restriction on the possible values of $\mu(x,y)$.

\begin{thm}
\label{correct2nd}
    Let $(X,d,\kappa)$ be a digital metric 
    space, such that $X$ is finite or $d$ is an $\ell_p$ metric. Let $T: X \to X$ be onto. Suppose
    \[ d(T(x), T(y)) \ge k \mu(x,y) \]
    for all $x,y \in X$, where $1<k<2$ and
    \[ \mu(x,y) \in \left \{d(x,y), 
    \frac{d(x,T(x)) + d(y,T(y))}{2}
    \right \}.
    \]
    Then $T$ has a fixed point.
\end{thm}

\begin{proof}
A proof can be given via suitable
modification of its analog in~\cite{Jyoti-Rani17}.
However, a simpler argument is as follows.

Without loss of generality, $|X| > 1$.
Since $X$ is finite or $d$ is an $\ell_p$ metric,
there exist $x_0,y_0 \in X$ such that
\[ m = d(x_0,y_0) = \min\{d(x,y) \, | \, x,y \in X,
       x \neq y\} > 0.
\]
Since $T$ is onto, there exist $x',y' \in X$ such
that $T(x')=x_0$ and $T(y')=y_0$.

Suppose $T$ has no fixed point. Then for all
$x,y \in X$, $d(x,T(x)) \ge m$ and 
$d(y, T(y)) \ge m$; hence $\mu(x,y) \ge m$.
Therefore, 
\[m=d(x_0,y_0) = d(T(x'), T(y')) \ge k \mu(x',y')
  \ge km,
\]
a contradiction. Therefore, $T$ must have a fixed
point.
\end{proof}

\subsection{3rd generalization of expansive mappings}
The next assertion of~\cite{Jyoti-Rani17} 
is flawed in ways similar to the assertion discussed in section~\ref{2ndGen}.
Asserted as Theorem~3.6 of~\cite{Jyoti-Rani17} is the following.
\begin{quote}
    Let $(X,d,\kappa)$ be a complete digital
    metric space.
    Let $T$ be an onto self-map of $X$ that
    is continuous. Let $k>1$ and suppose $T$
    satisfies
    \[ d(T(x),T(y)) \ge k \mu(x,y) \mbox{ for all } x,y \in X,
    \]
    where $\mu(x,y)$
    belongs to
    \[ \left \{ d(x,y), \frac{d(x,T(x))+d(y,T(y))}{2}, d(x,T(y)), d(y, T(x))
    \right \}.
    \]
    Then $T$ has a fixed point.
\end{quote}

Observe the following.
\begin{itemize}
    \item As above, the continuity used for the proof
          of this assertion is of the
          $\varepsilon - \delta$ kind, not digital 
          continuity.
    \item As above, the argument given in proof for this assertion requires $1 < k < 2$.
    \item As above, the authors attempt to establish a Cauchy sequence, and in doing so, they incorrectly reverse the triangle inequality
    in order to reduce the 3rd case considered to the 2nd case.
    \end{itemize}
    
Thus, as stated, the assertion presented as Theorem~3.6 of~\cite{Jyoti-Rani17} must be regarded 
as unproven. Note that Theorem~\ref{correct2nd} above
is a reasonable correct modification of 
this assertion.

\subsection{Examples  of~\cite{Jyoti-Rani17}}
In Examples 3.8, 3.9, 3.16, and 3.17 of~\cite{Jyoti-Rani17}, the authors lose track
of the standard assumption that a digital image $X$
is a subset of $\Z^n$. In each of these examples, they write of an unspecified $X$ using
functions that clearly place $X$ in $\R$, but not clearly in $\Z$.

\subsection{$\alpha - \psi$ expansive maps}
In the following, we let $\Psi$ be the set of functions~\cite{SametEtal}
$\psi: [0,\infty) \to [0,\infty)$ such that
\begin{itemize}
    \item $\sum_{n=1}^{\infty}\psi^n(t) < \infty$ for each $t>0$, where $\psi^n$ is the $n$-th iterate of $\psi$ (\cite{Jyoti-Rani17} misquotes this requirement as $\psi^n(t) < \infty$ for each $t>0$), and
    \item $\psi$ is non-decreasing.
\end{itemize}

\begin{definition}
\label{alphaPsiExpansion}
{\rm \cite{Jyoti-Rani17}}
Let $(X,d,\kappa)$ be a digital metric space.
Let $T: X \to X$. $T$ is a
{\em digital $\alpha-\psi$ expansive mapping} if
$\alpha: X \times X \to [0,\infty)$, 
$\psi \in \Psi$, and
for all $x,y \in X$,
\[ \psi(d(T(x), T(y))) \ge
\alpha(x,y) d(x,y).
\]
\end{definition}

\begin{definition}
\label{admissible}
{\rm \cite{Jyoti-Rani17}}
Let $T: X \to X$. Let
$\alpha: X \times X \to [0,\infty)$. 
$T$ is {\em $\alpha$-admissible} if $\alpha(x,y) \ge 1$ implies
$\alpha(T(x),T(y)) \ge 1$
\end{definition}

\begin{thm}
{\rm ~\cite{Jyoti-Rani17}}
\label{alphaPsiExpThm}
    Let $(X,d,\kappa)$ be a complete digital metric space. Let $T: X \to X$ be a bijective, digital
    $\alpha-\psi$ expansion mapping such that
    \begin{itemize}
        \item $T^{-1}$ is $\alpha$-admissible;
        \item There exists $x_0 \in X$ such that $\alpha(x_0,T^{-1}(x_0)) \ge 1$; and
        \item $T$ is digitally continuous.
    \end{itemize}
    Then $T$ has a fixed point.
\end{thm}

Despite the use of ``digitally continuous" in the statement
of Theorem~\ref{alphaPsiExpThm}, the continuity assumption used in its proof
is of the $\varepsilon-\delta$ variety.
In fact, the assumption is unnecessary if we assume additional common conditions, as in the following.

\begin{thm}
Let $(X,d,\kappa)$ be a digital metric space, 
where $X$ is finite or $d$ is an $\ell_p$ metric.
Let $T: X \to X$ be a bijective, digital
    $\alpha-\psi$ expansion mapping such that
    \begin{itemize}
        \item $T^{-1}$ is $\alpha$-admissible;
        \item There exists $x_0 \in X$ such that $\alpha(x_0,T^{-1}(x_0)) \ge 1$; and
    \end{itemize}
    Then $T$ has a fixed point.
\end{thm}

\begin{proof}
Our argument borrows from its analog in~\cite{Jyoti-Rani17}.

By hypothesis, there exists $x_0 \in X$
such that $\alpha(x_0,T^{-1}(x_0)) \ge 1$.
By induction, we obtain $S=\{x_n\}_{n=0}^{\infty}$ such that
$x_{n+1}=T^{-1}(x_n)$ for $n>0$.

Since $T^{-1}$ is $\alpha$-admissible, by induction
we have
\[\alpha(x_n,x_{n+1}) =
\alpha(T^{-1}(x_{n-1}),T^{-1}(x_n)) \ge 1
\]
for all $n$. Since $T$ is a
{\em digital $\alpha-\psi$ expansive mapping},
for all $n$ we have
\[ d(x_n,x_{n+1}) \le \alpha(x_n,x_{n+1}) d(x_n,x_{n+1}) \le
\psi(d(T(x_n), T(x_{n+1})) = \]
\[ \psi(d(x_{n-1}, x_n)).
\]

By induction, it follows that
$d(x_n,x_{n+1}) \le \psi^n(d(x_0,x_1))$.
Since $\psi \in \Psi$, it follows that $S$ is a Cauchy sequence. By
Theorem~\ref{eventuallyConst}, $S$
is eventually constant, so there
exists $m$ such that $T(x_{m+1})=x_m = x_{m+1}$; thus,
$x_{m+1}$ is a fixed point of $T$.
\end{proof}

\section{Weakly commuting mappings}
The paper~\cite{Jain-Upad17} presents a
fixed point assertion for ``weakly commuting
mappings," defined as follows.

\begin{definition}
{\rm \cite{Sessa}}
Let $(X,d)$ be a metric space and let
$f,g: X \to X$. Then $f$ and $g$ are
{\em weakly commuting} if for all $x \in X$,
$d(f(g(x)), g(f(x))) \le d(f(x), g(x))$.
\end{definition}

Presented as Theorem~3(A) of~\cite{Jain-Upad17}
is the following.

\begin{quote}
    Let $(X,d,\kappa)$ be a complete digital
    metric space, $X \neq \emptyset$. Let
    $S,T: X \to X$ such that
    
    (3.1) $T(X) \subset S(X)$;
    
    (3.2) $S$ is $\kappa$-continuous;
    
    (3.3) For some $\alpha$ such that 
    $0 < \alpha < 1$ and all $x,y \in X$,
    
    $d(T(x),T(y)) \le \alpha d(S(x), S(y))$.
    
    If $S$ and $T$ are weakly commuting, then
    they have a unique common fixed point.
\end{quote}

The argument given in proof of this assertion is
flawed by the unjustified statement (rephrased slightly), ``From~(3.2) the $\kappa$-continuity
of~$S$ implies the $\kappa$-continuity
of~$T$." This reasoning is incorrect, as shown in
the following.

\begin{exl}
Let $X=\{p_0=(0,0,0), p_1=(1,1,1),p_2=(2,0,0)\} \subset \Z^3$.
Let $S=1_X: X \to X$ and let $T:X \to X$ be defined
by $T(p_0)=T(p_2)=p_2$, $T(p_1)=p_0$. Let
$\kappa$ be the $c_3$-adjacency. Clearly,
(3.1) and (3.2) of the above are satisfied.
Let $d$ be the $\ell_1$ metric. Then (3.3) above is
satisfied with $\alpha = 2/3$. However,
$T$ is not $c_3$-continuous, since
$p_0 \adj_{c_3} p_1$ but $f(p_0)$ and $f(p_1)$ are
not $c_3$-adjacent.
\end{exl}

Therefore, the assertion stated as
Theorem~3(A) of~\cite{Jain-Upad17} must be
regarded as unproven.
However, we see below that replacing the assumptions
of completeness and~(3.2) by assumptions that are 
commonly realized yields a valid statement.

\begin{thm}
\label{correctJainUpad}
Let $(X,d,\kappa)$ be a digital
    metric space, $X \neq \emptyset$, with
    $X$ finite or $d$ an $\ell_p$ metric. Let
    $S,T: X \to X$ such that
\begin{enumerate}
\item $T(X) \subset S(X)$;
\item For some $\alpha$ such that 
    $0 < \alpha < 1$ and all $x,y \in X$,
    
    $d(T(x),T(y)) \le \alpha d(S(x), S(y))$.
\end{enumerate}
If $S$ and $T$ are weakly commuting, then
they have a unique common fixed point.
\end{thm}

\begin{proof}
We use ideas from the analogous argument
of~\cite{Jain-Upad17}.

Let $x_0 \in X$. By assumption 1, there exists
$x_1 \in X$ such that $S(x_1)=T(x_0)$. By
induction we have a sequence $\{x_n\}_{n=0}^{\infty}$ such that for all $n$,
$S(x_{n+1})=T(x_n)$, and we have
\[ d(S(x_n),S(x_{n+1})) = d(T(x_{n-1}), T(x_n))
   \le \alpha d(S(x_{n-1}), S(x_n)).
\]
By a simple induction, this yields
\[ d(S(x_n),S(x_{n+1})) \le \alpha^n d(S(x_0), S(x_1)).
\]
Using the Triangle Inequality with the latter, we
have, for any $p \in \N$,
\[ d(S(x_n), S(x_{n+p})) \le
   \sum_{i=n}^{n+p-1} d(S(x_i), S(x_{i+1}) \le
   \sum_{i=n}^{n+p-1} \alpha^i d(S(x_0), S(x_1)) \le
\]
\[ \frac{\alpha^n}{1-\alpha} d(S(x_0), S(x_1))
   \to_{n \to \infty} 0.
\]
Thus, $\{S(x_n)\}_{n=0}^{\infty}$ is a
Cauchy sequence, hence by Proposition~\ref{eventuallyConst} is eventually
constant, i.e., there exists $z \in X$ such that 
for sufficiently large~$n$,
\[ S(x_n) = z.
\]
By our definition of the sequence $\{x_n\}$, we
also have, for sufficiently large~$n$,
\[ T(x_n) = z.
\]
So for $n$ sufficiently large, and since $S$ and
$T$ are weakly commuting,
\[ d(S(z), T(z)) = d(S(T(x_n)), T(S(x_n))) \le d(S(x_n), T(x_n)) = d(z,z) = 0,
\]
i.e., $S(z)=T(z)$ and therefore, by the weakly
commuting property,
\[ d(S(T(z)), T(S(z))) \le d(S(z), T(z)) = 0,
\]
i.e., $S(T(z)) = T(S(z))$. So
\[ d(T(z), T(T(z))) \le \alpha d(S(z), S(T(z))) =
   \alpha d(T(z), T(S(z))) = \]
\[ \alpha d(T(z), T(T(z))).
\]
Thus $d(T(z), T(T(z)))=0$, i.e., $T(z)$ is a fixed
point of $T$. Further, substituting from the above
gives
\[ d(S(T(z)), T(z)) = d(T(S(z)), T(z)) \le
   \alpha d(S(S(z)), S(z)) = 
   \alpha d(S(T(z)), T(z));
\]
since $\alpha > 0$, it follows that
$d(S(T(z)), T(z)) = 0$.
Thus, $T(z)$ is a common fixed point of $S$ and $T$.

To show the common fixed point is unique, suppose
$y$ and $y'$ are common fixed points, i.e.,
\[ S(y) = T(y) = y, ~~~ S(y') = T(y') = y'.
\]
Then
\[ d(y,y') = d(T(y), T(y')) \le 
   \alpha d(S(y), S(y')) = \alpha d(y,y'),
\]
so $d(y,y')=0$. Hence $y=y'$.
\end{proof}

Note the following limitation on
Theorem~\ref{correctJainUpad} is applicable if
$X$ is finite, or if $d$ is an $\ell_p$ metric.

\begin{prop}
Let $(X,d,\kappa), S, T, \alpha$ be as in
Theorem~\ref{correctJainUpad}, where
\[ d_0 = \min \{d(x,x') \, | \,
    x, x' \in X, x \adj_{\kappa} x' \} > 0,
\]
\[ d_1 = \max \{d(x,x') \, | \,
    x, x' \in X, x \adj_{\kappa} x' \}.
\]
If $X$ is $\kappa$-connected, $S$ is 
$\kappa$-continuous, and $0<\alpha < d_0 / d_1$, then $T$ is a constant function.
\end{prop}

\begin{proof}
Let $x \adj_{\kappa} x'$. Since $S$ is 
$\kappa$-continuous we have
$S(x) \adjeq_{\kappa} S(x')$, and therefore
$d(S(x), S(x')) \le d_1$. Thus,
\[ d(T(x), T(x')) \le \alpha d(S(x),S(x')) < 
   \frac{d_0}{d_1} d_1 = d_0.
\]
Our choice of $d_0$ implies $T(x)=T(x')$. Since
$X$ is connected, the assertion follows.
\end{proof}


\section{Weakly compatible maps}
The paper~\cite{Dalal17} discusses ``weakly
compatible" or ``coincidentally commuting" maps,
defined as follows.

\begin{definition}
Let $S,T: X \to X$. Then $S$ and $T$ are
{\em weakly compatible} or 
{\em coincidentally commuting} if, for every 
$x \in X$ such that $S(x)=T(x)$ we have 
$S(T(x)) = T(S(x))$.
\end{definition}

The following assertion is stated as Theorem~3.1
of~\cite{Dalal17}.

\begin{quote}
    Let $A,B,S,T: X \to X$, where $(X,d,\kappa)$
    is a complete digital metric space. Suppose
    the following are satisfied.
    \begin{itemize}
        \item $S(X) \subset B(X)$ and $T(X) \subset A(X)$.
        \item The pairs $(A,S)$ and $(B,T)$ are
              coincidentally commuting.
        \item One of $S(X),T(X),A(X),B(X)$ is a
              complete subspace of $X$.
        \item For all $x,y \in X$, $d(S(x),T(y)) \le$
        \[ \phi(\max\{d(A(x),B(y)), d(S(x),A(x)),       d(S(x),B(y)), d(B(y),T(y))\})
        \]
        where $\phi: [0,\infty) \to [0,\infty)$ is
        continuous, monotone increasing, and
        satisfies $\phi(t) < t$ for all $t>0$.
    \end{itemize}
    Then $A,B,S$, and $T$ have a unique common
    fixed point.
\end{quote}

However, the argument offered as proof of this
assertion is flawed as follows. A sequence
$\{y_n\}_{n=0}^{\infty}$ is established and it
is shown that 
$\lim_{n \to \infty} d(y_{2n},y_{2n+1})=0$.
From this, it is claimed that
$\{y_n\}_{n=0}^{\infty}$ is a Cauchy sequence. But
such reasoning is incorrect, as shown in the following.

\begin{exl} For $n \in \N$, let
\[ y_n = \left \{ \begin{array}{ll}
         0 & \mbox{if } (n \mod 4) \in \{0,1\}; \\
         1 & \mbox{if } (n \mod 4) \in \{2,3\},
         \end{array} \right .
\]
For all $n \in \N$, $d(y_{2n},y_{2n+1})=0$, yet 
$\{y_n\}_{n=0}^{\infty}$ is not a Cauchy sequence.
\end{exl}

Thus, the assertion of~\cite{Dalal17} stated 
as Theorem~3.1, and its dependent assertion stated 
as Theorem~3.2, must be regarded as unproven.

\section{Further remarks}
We have discussed assertions that appeared 
in~\cite{BEKLL,Dalal17,Jain-Upad17,Jyoti-Rani17}. 
We have discussed errors or corrections for some,
shown some to be limited or trivial, and offered 
improvements for others.

\end{document}